\documentclass[12pt]{huber_article}

\usepackage{amsmath,amssymb,amsthm,array,comment,setspace,url} 
\usepackage{graphicx}
\usepackage{algorithm}
\usepackage{algorithmic}

\newcommand{\INPUT}{\REQUIRE}  
\newcommand{\OUTPUT}{\ENSURE}

\newcommand{\prob}{\mathbb{P}}

\newcommand{\bern}{\operatorname{Bern}}

\newcommand{\ind}{{\bf 1}}

\usepackage{fullpage}

\usepackage{amsmath,amsfonts,amsthm}

\newcommand{\vecBeta}{\mathbf{\beta}}
\newcommand{\vecB}{B}

\newcommand{\subs}{\Omega_{\operatorname{subs}}}
\newcommand{\rc}{\Omega_{\operatorname{rc}}}
\newcommand{\pispins}{\pi_{\operatorname{spins}}}
\newcommand{\pisubs}{\pi_{\operatorname{subs}}}
\newcommand{\pirc}{\pi_{\operatorname{rc}}}

\newcommand{\wsubs}{w_{\operatorname{subs}}}

\newcommand{\Zrc}{Z_{\operatorname{rc}}}

\newcommand{\degree}{{\operatorname{deg}}}
\newcommand{\calI}{{\cal I}}
\newcommand{\lame}[1]{\lambda_{e \rightarrow #1}}

\newcommand{\cluster}{{\cal C}}
\newcommand{\on}[1]{{\operatorname{#1}}}

\newtheorem{thm}{Theorem}

\begin{document}

\title{Simulation reductions for the Ising model}

\maketitle

{\bfseries \sffamily Mark L. Huber} \par
{\slshape Department of Mathematics and Computer Science, 
  Claremont McKenna College} \par
{\slshape mhuber@cmc.edu}
\vskip 1em

\abstract{Polynomial time reductions between problems have long been
used to delineate problem classes.  Simulation reductions also exist,
where an oracle for simulation from some probability distribution
can be employed 
together with an oracle for Bernoulli draws in order to obtain
a draw from a different distribution.  Here linear time
simulation reductions are given
for: the Ising spins world to the Ising subgraphs world and
the Ising subgraphs world to the Ising spins world.
This answers a long standing question of whether such a direct relationship
between these two versions of the Ising model existed.  Moreover,
these reductions result in the first method
for perfect simulation from the subgraphs world and a new Swendsen-Wang style
Markov chain for the Ising model.
The method used is to write the desired distribution with
set parameters as a mixture of distributions where the 
parameters are at their extreme values.}

\section{Introduction}

In this paper, two different forms of the Ising model are linked via
a polynomial time simulation reduction.  Consider a family of probability 
distributions $\pi$ parametrized by inputs $\calI$, together with 
another family $\pi'$.
Say that $\pi$ is 
simulation reducible to $\pi'$ if an algorithm exists for drawing from
$\pi$ on $\calI$ that is allowed to take draws from $\pi'$ as well
as utilize extra randomness in the form of independent Bernoulli draws whose
parameter is determined at runtime.  In other words, if the ability to draw from
$\pi'$ together with the ability to flip coins
with arbitrary probabilities of coming up heads is enough to draw a sample
from $\pi$, then $\pi$ is simulation reducible to $\pi'$.  
As with problem reductions, 
sampling reductions are of most interest when they are polynomial, that is,
when the sum of the sizes of the inputs to $\pi'$ plus the number
of Bernoullis used is a polynomial in the size of $\calI$.

A famous example of simulation reductions appears in the Swendsen-Wang
approach~\cite{swendsenw1987} to the Ising model.  
The original version of the Ising model
used spins on nodes in graphs (call this the \emph{spins} view) while
another used subsets of edges in graphs (call this the \emph{random cluster}
view).  As part of the Swendsen-Wang algorithm, a draw from the random cluster
model could be turned (using a number of Bernoullis equal to the number of
clusters formed by the edges) into a spins draw.  Similarly, a spins draw
could be turned (using a number of Bernoullis equal to the number of edges
in the graph) into a random cluster draw.

In this work, a third view of the Ising model is dealt with.  This third
view also uses subsets of edges, but assigns them weights that are very
different from those of the random cluster model.  
Following~\cite{jerrums1993}, this third model will be called the
\emph{subgraphs} view of the Ising model.

The families of distributions corresponding to spins, random clusters,
and subgraphs will be denoted $\pispins$, $\pirc$, and $\pisubs$ respectively.
The remainder of the paper is organized as follows.  In 
Section~\ref{SEC:Ising}, the three families of the Ising model are described
in detail.  Section~\ref{SEC:theory} discusses the framework that will be
used for the reductions.  Section~\ref{SEC:subs2rc} shows how a single 
draw from the subgraphs model together with a number of Bernoullis at
most equal to the number of edges can be used to create a 
draw from the random cluster model.  Section~\ref{SEC:rc2subs} then shows
the reverse direction:  how a single draw from the random cluster model
together with a number of Bernoullis equal to the number of edges can
be used to draw from the subgraphs model.  Since the reductions between
$\pirc$ and $\pispins$ are well known, and require a single
sample each, it is possible to take a single draw from $\pisubs$ and
convert it to a single draw from $\pispins$ with a small number of 
Bernoullis.  Similarly, a single draw from $\pispins$ can be converted
to a single draw from $\pisubs$ with a small number of Bernoullis.  This
yields a new Swendsen-Wang style Markov chain for approximately sampling
from the Ising model.
Also, it
answers a long standing question in~\cite{jerrums1993} of whether or not
such a direct relationship between the distributions $\pispins$ and 
$\pisubs$ existed.  These results are discussed in detail in 
Section~\ref{SEC:discuss}.

Earlier work in \cite{randallw1999} created an approximate simulation 
reduction from $\pispins$ to $\pisubs$ by creating a series of 
approximations of partition functions for multiple inputs.  This method has
the drawback of requiring multiple draws from $\pisubs$ in order to 
obtain a single configuration that is approximately drawn from $\pispins$.  
In contrast, the reduction presented here is exact, 
and requires only one draw from
$\pisubs$ to generate one draw from $\pispins$ or vice versa.  
Because the reduction is exact, 
this provides for the first perfect simulation algorithm for $\pisubs$,
the details are also
in Section~\ref{SEC:discuss}.

\section{The Ising model}
\label{SEC:Ising}
The Ising model was initially proposed as a model for 
magnitization and has been studied extensively becuase of the presence 
of a phase transition in lattices of dimension two or higher~\cite{simon1993}.  
It has also been employed in statistical applications including agricultural
studies and image restoration~\cite{besag1974}.

\paragraph{Spins model}
In the ferromagnetic Ising model 
each node of a graph $G = (V,E)$ is assigned
either a 1 or -1.  Because of the original use as a model of 
magnetism, the 1 nodes are
referred to as ``spin up'' and the -1 nodes are called ``spin down''.

A state $x \in \{-1,1\}^V$ is called a \emph{configuration}.  The
weight of a configuration is of the form:
\[
w_\on{spins}(x) = \prod_{\{i,j\} \in E} f(x(i),x(j)) 
\]
For a weight function $w_\on{spins}(x)$, the probability distribution is 
$\pi_\on{spins}(x) = w_\on{spins}(x) / Z_\on{spins}$, 
where $Z_\on{spins} = \sum_{x \in \{-1,1\}^V} w_\on{spins}(x)$ 
is called the \emph{partition function}.

Let 
$\vecBeta$ be a nonnegative vector over the edges $E$, and 
$\vecB$ be a nonnegative vector over the nodes $V$.  Then 
\begin{equation}
\label{EQN:f_definition}
f(x(i),x(j)) = \left\{ \begin{array}{ll} 
  \exp(\beta(\{i,j\}) x(i) x(j)) & \beta(\{i,j\}) < \infty \\
  \ind(x(i) = x(j)) & \beta(\{i,j\}) = \infty,
   \end{array}              \right. 
\end{equation}
Hence $\beta$ controls the strength of interaction between spins at
endpoints of edges.  When $\beta$ is infinite, the attraction is infinite,
and endpoints must be given
the same value.  

An extension of the basic Ising model is to allow for an external
magnetic field.  In this case, the weight function acquires factors
based directly on the value of each node:
\[
w_\on{spins + field}(x) = 
  \left[ \prod_{\{i,j\} \in E} f(x(i),x(j)) \right]
  \left[ \prod_{i \in V} g(x(i)) \right],
\]
where $g$ is paramerized by a vector $B$ that controls the strength
of the magnetic field:
\[
g(x(i)) = \left\{ \begin{array}{ll} 
  \exp(\vecB(i) x(i)) & \on{for }\ \ |\vecB(i)| < \infty \\
  \ind(x(i) = 1) &  \on{for }\ \ \vecB(i) = \infty \\
  \ind(x(i) = -1) & \on{for }\ \ \vecB(i) = -\infty \\

   \end{array}              \right. 
\]

Say that the magnetic field is \emph{unidirectional} if $B \geq 0$
or $B \leq 0$.  The unidirectional model is easier than the model with
general magnetic field in the sense that it can be easily reduced 
to the case with no magnetic field as follows.
Suppose without loss of generality $B \geq 0$.  

If $B(v) < \infty$ for all $v$, then build a new
graph by creating
a dummy node $v_B$ and for all $i \in V$ adding an edge 
$\{i,v_B\}$ with edge weight $\beta(\{i,v_B\}) = (1/2) B(i).$  Note 
that any configuration $x$ satisfying $x(v_B) = 1$ has 
\[ w_{\textrm{spins + field}}(x(V)) = 
  w_{\textrm{spins only}}(x(V))
  \prod_{i \in V:B(i) < \infty} \exp((1/2) B(i)). \]
Since the weights differ by only a constant, the probability distributions
are the same.

Moreover, under $w_\textrm{spins only}$ the weight of $x$ and $-x$ are the
same, so taking a draw from $w_\on{spins}$ and choosing $x$ or $-x$
so that $v_B$ has value 1 is an easy way to draw from 
$w_\on{spins}$ conditioned on $v_B$ being 1.  Hence the unidirectional
magnetic field can be brought into the case of no magnetic field with
no loss of generality when $B(v) < \infty$ for all $v$.

When there exist one or more nodes with $B(v) = \infty$, there is
no need to create the dummy node $v_B$.  These nodes all must be spin
up, and so they can be merged, and then the merged node can be 
used in the same fashion as $v_B$ above.

\paragraph{Subgraphs model}
In a seminal paper Jerrum and Sinclair \cite{jerrums1993} 
showed how to approximate
$Z_\on{spins}(\vecBeta)$ for any graph and any $\vecBeta$ 
in polynomial time.  Their approach did not tackle the spins
model directly.
Instead, they developed results for 
a different formulation of the Ising model:  the subgraphs model.
This model is also a Markov random field, but assigns values to the edges
of the graph rather than the nodes.

Let $\subs = \{0,1\}^E$, and now a configuration $y \in \subs$
can be viewed as
a collection of edges of the graph, where $y(e) = 1$ indicates
that the edge is in the collection and otherwise it is not.  
Since each $y$ encodes a subset of edges, this was called the 
subgraphs world in~\cite{jerrums1993}.
Define $\lambda$ over the 
edges as follows:
\begin{equation}
\label{EQN:lambda_definition}
\lambda(\{i,j\}) = \tanh \beta(\{i,j\})
\end{equation}
where $\tanh \infty$ is taken to be $1$.
Then in the subgraphs model, the weight function is
\begin{equation}
w_\on{subs}(y) = \left[ \prod_{e:y(e) = 1} \lambda(e) \right] 
  \left[ \prod_i \ind(\on{deg}(i,y) \text{ is odd}) \right].
\end{equation}
The degree of a node $i$ in subgraph configuration $y$ is 
\begin{equation}
\degree(i,y) := \sum_{j \neq i} y(\{i,j\}),
\end{equation}
and $\ind(A)$ is the indicator function that has value 1 if the
Boolean expression $A$ is true, and is 0 otherwise.

As with the spins model, $\pi_\on{subs}(y) = w_\on{subs}(y) / Z_\on{subs}$,
where $Z_\on{subs}$ is the sum of the weights over all configurations, or
equivalently the normalizing constant 
that makes $\pisubs$ a probability
distribution.  
The subgraphs model is also known as the {\em high temperature expansion}, and
was introduced by van der Waerden~\cite{vanderwaerden1941}.  Like the
spins model, here the weight depends on the product of individual factors
that depend either on a single edge or only on the edges leaving a particular
node.

One note:  in the Jerrum and Sinclair~\cite{jerrums1993} paper 
the formulation of the subgraphs world 
included terms that allowed for a unidirectional
magnetic field.  As noted in the previous section, a unidirectional
field can be eliminated from the problem without loss of generality,
and so here the simpler form of the subgraphs model is used.

\paragraph{Random Cluster model}
A third approach to the Ising model was introduced by Fortuin and 
Kasteleyn~\cite{fortuink1972}.  
Unlike the spins and subgraphs views of the Ising model, 
the random cluster model
is decidedly nonlocal in its weight function.

Like the subgraphs model, the state space $\rc = \{0,1\}^E$ indexes
a collection of edges.  These edges partition the nodes $V$ into
the set of maximally connected components, known as \emph{clusters}.
To be precise, consider a configuration $z \in \rc$.
A collection of nodes $C$ is a cluster in $z$ if for all
$v$ and $v'$ in $C$, there is a path 
$\{v = v_1,v_2,v_3,\ldots,v_n = v'\}$
such that $\{v_{i},v_{i+1}\} \in E$ and $z(\{v_{i},v_{i+1}\}) = 1$ for all $i$.
Moreover, for all $v \in C$ and $v'' \notin C$,
either $\{v,v''\} \notin E$ or 
$z(\{v,v''\}) = 0.$  Let $\mathcal{C}$ denote the set of clusters.

The weight function for the random cluster model is:
\[w_\on{rc}(z) = \left[\prod_{e:z(e) = 1} p(e)\right] 
  \left[ \prod_{e:z(e) = 0} (1 - p(e)) \right] 
  2^{\#\mathcal{C}}.
\]
where
\begin{equation}
\label{EQN:p_definition}
p(e) = 1 - \exp(-2\beta(e)),
\end{equation}
and $\exp(-\infty)$ is taken to be 0.

The procedure for generating a spins draw from a random cluster draw
is as follows:  independently 
for each cluster, draw uniformly 
from $\{-1,1\}$ and assign all nodes in that cluster the randomly
chosen value.  (See~\cite{swendsenw1987} for further details.)  

The idea behind this is as follows.  The spins model with parameter
$\beta$ can be viewed
as a mixture of many different spins models, each with 
a different set of parameters.
Each component of the mixture
has a parameter vector whose $\beta$ values for every edge are either
$0$ or $\infty$, and so altogether there are $2^{\#E}$ different components
of the mixture.  The random cluster configuration indexes which 
component of the mixture to use:  $z(e) = 1$ indicates $\beta(e) = \infty$,
while $z(e) = 0$ indicates $\beta(e) = 0.$  Drawing a spins model with
such a simple parameter vector is easy:  all nodes connected by edges with
$\beta(e) = \infty$ must have the same spin, while the $\beta(e) = 0$ 
edges might as well be gone from the graph.

This idea of using an index from a mixture is 
developed formally in Section~\ref{SEC:theory}.

\paragraph{Relationship between spins and subgraphs}
The normalizng 
constants $Z_\on{spins}$ and $Z_\on{subs}$ are related by an easy to calculate
constant.  The following result goes back to~\cite{newellm1953}.

\begin{thm}
 \label{THM:relate}
\begin{equation}
\label{EQN:relate}
Z_\on{spins} = Z_\on{subs} 
   2^{\#V} \prod_{e} \cosh \beta(e).
\end{equation}
\end{thm}

While this result is straightforward to show analytically, such a proof
does not offer much probabilistic 
insight into why this remarkable relationship is
true.  In~\cite{jerrums1993} it is noted that ``there is no direct 
correspondence between configurations in the two domains and the
subgraph configurations have no obvious physical significance''.  
In Section~\ref{SEC:discuss} a new proof of this result is presented
that is based on the simulation reductions presented here.

In~\cite{jerrums1993}, a Markov chain that moved among configurations
in the subgraphs world was developed.  More importantly, this chain was 
shown to be rapidly mixing for all graphs.  
This result could then be used with the idea
of selfreducibility~\cite{jerrumvv1986} to obtain a 
fully polynomial approximation scheme (fpras) for $Z_\on{subs}$, and hence
$Z_\on{spins}$ as well.

\section{Drawing from mixtures}
\label{SEC:theory}

The reductions between spins and subgraphs draws 
come from viewing a distribution $\pi$ as a 
convex mixture of
several distributions, so
$$\pi = \alpha_1 \pi_1 + \cdots \alpha_M \pi_M,$$
where $\alpha_1 + \cdots \alpha_M = 1$.

Typically the parameters for distributions of interest (like the Ising
model) form a convex set.  The mixture of $\pi$ uses distributions 
where the parameter values are at their extreme values.  Usually
these extreme values of parameters correspond to simpler distributions.
Many times, in reducing from $\pi$ to $\pi'$, 
$\pi' = \alpha_1 \pi'_1 + \cdots \alpha_M \pi'_M$ where the 
$\alpha_i$ are equal to the cofficients for $\pi$.  Therefore, the algorithm
proceeds as follows:
\renewcommand{\labelenumi}{(\arabic{enumi})}
\begin{enumerate}
\item{Draw a sample from $\pi$,}
\item{Choose which $\pi_i$ the sample came from,}
\item{Draw a sample from $\pi_i'$,}
\item{Return as a sample from $\pi'$.}
\end{enumerate}

This can be formulated as a special case of the auxilliary variable method.
Let $I$ be a random variable where $\prob(I = i) = \alpha_i$ (this extra
variable $I$ is the auxilliary variable.)  
If $X|I \sim \pi_I$, then $X \sim \pi$.  The algorithm takes advantage
of this in reverse:  Given $X \sim \pi$ begin by choosing $I|X$.
Then choose $X'|I$ from $\pi'$.  Then $X' \sim \pi'$ as desired.

Typically, direct computation of the $\alpha_i$ is difficult, 
which is why it is not possible to just draw a random $I$ directly.
However, drawing $I|X$ is possible, by proceeding through several
stages.

Begin by considering a mixture of two distributions, so
$\pi = \alpha_1 \pi_1 + \alpha_2 \pi_2$.  At this first stage, choose
$I(1) \in \{1,2\}$ given $X$.  Then break $\pi_{I(1)}$ into two 
distributions, and so on until the index $I = (I(1),I(2),\ldots,I(k))$ has
been chosen.  Typically the stages are set up so that $\pi'_I$ is easy
to sample from.  The following theorem provides the algorithm for choosing
$I(a)$ given $X$ from the two choices.

\begin{thm} 
\label{THM:mixture}
Consider three distributions $\pi$, $\pi_1$ and $\pi_2$, respectively
defined by unnormalized weight functions $w$, $w_1$ and $w_2$ with the
appropriate normalizing constants $Z = \sum_y w(y)$, $Z_1 = \sum_y w_1(y)$,
and $Z_2 = \sum_y w_2(y).$  Suppose that $w(x) = c_1 w_1(x) + c_2 w_2(x),$
where the $c_i$ are positive constants.

Then $\pi$ is a mixture of $\pi_1$ and $\pi_2$, so 
$\pi = \alpha_1 \pi_1 + \alpha_2 \pi_2$.  Let $X \sim \pi$.  
Let $I | X$ be a binary random variable that is 1 with probability
$c_1 w_1(X) / [c_1 w_1(X) + c_2 w_2(X)]$ and is $2$ otherwise.
Then $\prob(I = 1) = \alpha_1$, $\prob(I = 2) = \alpha_2$, and 
$[X | I] \sim \pi_I$.
\end{thm}

\begin{proof}
Start with $w(x) = c_1 w_1(x) + c_2 w_2(x)$ and divide through
by $Z$.  This yields $\pi(x)$ on the left hand side.
Multiply and divide term $i$ on the right hand side by $Z_i$ to obtain:
$$\pi(x) = c_1 \pi_1 (x) (Z_1 / Z) + c_2 \pi_2(x) (Z_2 / Z).$$
Hence $\alpha_1 = c_1 Z_1 / Z$ and $\alpha_2 = c_2 Z_2 / Z$.
Also,
$$\prob(I = i) = \sum_y \prob(I = i|X = y)\pi(y) = 
  \sum_y \frac{c_i w_i(y)}{w(y)} \frac{w(y)}{Z} = 
  \frac{c_i}{Z} \sum_{y} w_i(y) = \frac{c_i Z_i}{Z} = \alpha_i.$$

Turning it around:
$$\prob(X = y|I = i) = \prob(I = i|X = y) \pi(y) / \alpha_i 
 = \frac{c_i w_i(y)}{w(y)} \frac{w(y)}{Z} \frac{Z}{c_i Z_i} 
   = \frac{w_i(y)}{Z_i}.$$
So $[X|I = i]$ has distribution $\pi_i$, completing the proof.
\end{proof}

This theorem can be easily extended to an arbitrary number of distributions,
however, mixtures of $\pi_1$ and $\pi_2$ 
suffice for all the results in this paper.

\section{Subgraphs to random cluster}
\label{SEC:subs2rc}

In this section it is shown how to utilize a subgraphs draw together
with at most $\#E$ Bernoulli draws to generate
a draw from the random cluster model.  This can then be used to create
a spins draw if desired.

Let $\pisubs$ be the distribution parameterized by $\lambda$.
Consider any edge $e$ with $\lambda(e) \in (0,1)$.  Create a new
vector $\lambda_{e \rightarrow 1}$ 
that takes on value 1 on edge $e$, and matches
$\lambda$ at all other edges.
Similarly, let $\lambda_{e \rightarrow 0}$ equal $\lambda$ on all
edges but $e$, and let $\lambda_{e \rightarrow 0}(e) = 0.$

Then since $\lambda(e) \in (0,1)$:
\begin{equation}
\label{EQN:weight}
\wsubs(y ; \lambda) = \lambda(e) \wsubs(y ; \lame1) + 
  (1 - \lambda(e)) \wsubs(y ; \lame0),
\end{equation}

This has the exact form needed for Theorem~\ref{THM:mixture}.  The algorithm
is simple.  Begin with a draw $Y$ from $\pi$ parameterized with $\lambda$.
Calculate $\lambda(e)$ times the weight of $Y$ under 
$\lambda_{e \rightarrow 1}$, and 
$(1 - \lambda(e))$ times the weight of $Y$ under 
$\lambda_{e \rightarrow 0}$.  Draw
index $I$ from $\{1,0\}$ with probability proportional to these two numbers,
then set $\lambda(e)$ to $I$.

This means 
$\prob(I = 1) = \lambda(e) \wsubs(Y;\lame{1}) / \wsubs(Y;\lambda).$
When $Y(e) = 1$, this is just 1.  When $Y(e) = 0$, this becomes $\lambda(e)$.
Therefore
$$\prob(I = 1) = \lambda(e) + (1 - \lambda(e))Y(e).$$
After choosing the value of $I$, 
$Y$ given $I$ comes from the correct half of the mixture, and so the
process can be repeated again, until all of the entries in $\lambda$ are
either 0 or 1.  

\begin{algorithm}
\caption{{\tt Reduce edge weights}}
\label{ALG:reduce_edges}
\begin{algorithmic}[1]
\INPUT  parameter $\lambda$, $Y \sim \pisubs(\cdot;\lambda))$.
\OUTPUT new edge parameters $\lambda$
  \FOR {all edges $e$ with $\lambda(e) \in (0,1)$}
    \STATE {\bf draw} $\lambda(e) \leftarrow \bern 
      [\lambda(e) + (1 - \lambda(e))Y(e)]$
  \ENDFOR 
\end{algorithmic}
\end{algorithm}

Note that when $Y(e) = 1$, then 
$\lambda(e)
 \wsubs(Y;\lame{1}) / \wsubs(Y;\lambda) = 1,$ so the 
algorithm always sets $\lambda(e)$ to 1 in this case.  Only when 
$Y(e) = 0$ is there a choice with the possibility of $\lambda(e) = 0$.

At the end of the this algorithm, the draw $Y$ has been slotted into 
one of $2^{\#E}$ components of a mixture indexed by the new parameter
vector $\lambda$.  Each element of $\lambda$ is now either 0 or 1,
in the spins world this corresponds to $\beta$ being either 0 or 
$\infty$.  Hence the two nodes are either independent of each other or
forced to be the same, just as in the random cluster model.

This means, if the random output of Algorithm~\ref{ALG:reduce_edges} is called
$\Lambda$, then $\Lambda$ is a draw from 
the random cluster model of Fortuin and 
Kasteleyn~\cite{fortuink1972}.  This is stated precisely in the following
theorem.

\begin{thm}
\label{THM:subs2rc}
Let $Y \sim \pisubs(\cdot;\lambda)$.  Then
let $\Lambda$ be the output of Algorithm~\ref{ALG:reduce_edges} with
input $(\lambda,Y)$.  Let 
$p(e) = 1 - \exp(-2\beta(e)) = 2 \lambda(e)/(1 + \lambda(e))$.  Then 
$$\prob(\Lambda = z) = \frac{1}{\Zrc} 
 \left[ \prod_{e:z(e) = 1} p(e) \right]
 \left[ \prod_{e:z(e) = 0} 1 - p(e) \right]
 2^{\# {\cluster}(z)},$$
where
$$\Zrc = Z_\on{subs} 2^{\#V - \#E} \prod_e (1 + \exp(-2\beta(e)).$$
\end{thm}

\begin{proof}  Fix $z \in \{0,1\}^E$.  Consider a particular cluster in $z$.
Suppose that $Y = y$, and consider the chance that $\Lambda = z$ given 
$Y = y$.  Edges with $y(e) = 1$ always have $\Lambda(e) = 1$, so if
$z(e) = 0$ then $y(e)$ must be zero to have positive chance of $\Lambda = z$.
That is, $y \leq z$.

Now consider how a random draw could result in $\Lambda = z$ given $Y = y$.
For each edge with $y(e) = 0$ and $z(e) = 1$, there is a $\lambda(e)$ chance
of making this choice.  Hence the chance these edges agree is 
$\prod_{e:y(e) = 0,z(e) = 1} \lambda(e).$  For each edge $e$ with $y(e) = 1$
and $z(e) = 1$, the chance these edges agree is 1.  Finally, if 
$y(e) = 0$ and $z(e) = 0$, the chances that these agree is 
$\prod_{e:y(e) = 0,z(e) = 0} 1 - \lambda(e).$  Therefore,
\[
\prob(\Lambda = z|Y = y) = 
  \left[ \prod_{e:y(e) = 0,z(e) = 1} \lambda(e) \right]
  \left[ \prod_{e:y(e) = 0,z(e) = 0} 1 - \lambda(e) \right] \\
\]

Now consider the probability that
$Y = y$.  Since $Y$ is a subgraphs draw, 
this probability
is 0 unless all the nodes under $y$ have even degree, in which case the
probability of choosing $y$ is 
$Z_\on{subs}^{-1} \prod_{y(e) = 1} \lambda(e).$  This equals
$Z_\on{subs}^{-1} \prod_{y(e) = 1,z(e) = 1}$ since 
$y(e) = 1$ implies $z(e) = 1$ as well.  Therefore for all 
$y$ where each node has even degree
\begin{eqnarray*}
\prob(Y = y,\Lambda = z) & = &  Z_\on{subs}^{-1}
  \left[ \prod_{e:y(e) = 1,z(e) = 1} \lambda(e) \right] 
  \left[ \prod_{e:y(e) = 0,z(e) = 1} \lambda(e) \right]
  \left[ \prod_{e:y(e) = 0,z(e) = 0} 1 - \lambda(e) \right] \\
 & = & Z_\on{subs}^{-1} \left[ \prod_{e:z(e) = 1} \lambda(e) \right]
       \left[ \prod_{e:z(e) = 0} 1 - \lambda(e) \right] \\
\end{eqnarray*}
To find $\prob(\Lambda = z)$ all that remains is to sum over all
$y$ where each node has even degree.  Since 
$\prob(Y = y,\Lambda = z)$
is independent of $y$,
\begin{equation}
\label{EQN:inproof}
\prob(\Lambda = z) = \#\{y:y \leq z,\degree(i,y) 
 \text{ even } \forall i\} 
Z_\on{subs}^{-1} \left[ \prod_{e:z(e) = 1} \lambda(e) \right]
       \left[ \prod_{e:z(e) = 0} 1 - \lambda(e) \right] .
\end{equation}

So now consider how many $y \leq z$ states have even degree everywhere.
Consider a particular cluster $C$ in $z$, and let $C_E$ denote the set
of edges between nodes
in $C$.
Since the nodes are connected,
there exists a tree $T$ with edges $T_E \subseteq C_E$. 
Let $A = C_E \setminus T_E$ be those edges in the cluster
that are not part of the tree.

The key fact about $A$ is that
for any value of $y(A)$ there exists exactly one value of $y(T_E)$ that
ensures that $y(T_E \sqcup A)$ has even degree at each node.
See Figure~\ref{FIG:example} for an example.

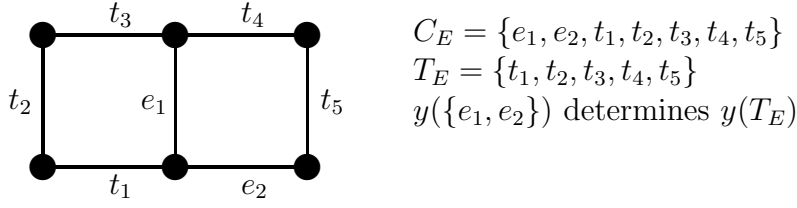
\begin{figure}[ht]
\begin{picture}(300,100)(-40,0)
  \put(50,20){\circle*{10}}
  \put(100,20){\circle*{10}}
  \put(150,20){\circle*{10}}
  \put(50,70){\circle*{10}}
  \put(100,70){\circle*{10}}
  \put(150,70){\circle*{10}}
  \thicklines
  \put(50,20){\line(0,1){50}}
  \put(75,10){$t_1$}
  \put(100,20){\line(0,1){50}}
  \put(125,10){$e_2$}
  \put(150,20){\line(0,1){50}}
  \put(155,42){$t_5$}
  \put(50,20){\line(1,0){50}}
  \put(37,42){$t_2$}
  \put(100,20){\line(1,0){50}}
  \put(87,42){$e_1$}
  \put(50,70){\line(1,0){50}}
  \put(75,75){$t_3$}
  \put(100,70){\line(1,0){50}}
  \put(125,75){$t_4$}
  \put(190,68){$C_E = \{e_1,e_2,t_1,t_2,t_3,t_4,t_5\}$}
  \put(190,53){$T_E = \{t_1,t_2,t_3,t_4,t_5\}$}
  \put(190,38){$y(\{e_1,e_2\})$ determines $y(T_E)$}
\end{picture}
\caption{Example of cluster with six nodes}
\label{FIG:example}
\end{figure}

To show this,
fix $y(A)$.
Start with
a leaf $i$ of the tree $T$.  
There is exactly one edge in $T_E$ adjacent to $i$ with
$z(e) = 1$, so there is precisely one way to choose $y(e)$ in order to 
maintain that the degree of $i$ must be even under $y(e)$.  Now consider
the remaining edges in $T_E$.  The nodes connected to these
edges forms a smaller tree, so again a leaf
exists, and again the there is exactly one way to choose $x$ on this edge
to maintain the even degree requirement.  
After $\#T$ steps the unique value for $y(T)$ that preserves the even
degree at each node will have been found.

So given $y(A)$, there is exactly one choice of $y$ over the edges of the
cluster where all edges of the cluster receive even degree.
There are $2^{\#A}$ different possible values of 
$y(A)$, and so there are $2^{\#A}$ different choices of $y$ over
the cluster.  The number of edges in $A$ is the number of edges in the cluster
minus the number of edges in the tree, so for a cluster $C \in {\cal C}$,
this is
\[
\#A = \left[ \sum_{e \in C_E} z(e) \right] 
   - [\# C - 1].
\]

The choice of $y(C_E)$ is independent for each cluster $C \in \cluster$, and
so the total number of $x$ such that $x \leq z$ and each node of 
$x$ has even degree is the product
$$\prod_{C \in \cluster} 2^{\sum_{e \in C} z(e) - 
  [\# C - 1]}
  = 2^{\left[\sum_{e \in C} z(e)\right] - \#V + \#\cluster(z)}.$$
Combining this with (\ref{EQN:inproof}) yields
$$\prob(\Lambda = z) = Z_\on{subs}^{-1} 2^{-\#V + \#\cluster(z) } 
                       \left[ \prod_{e:z(e) = 1} 2 \lambda(e) \right] 
                       \left[ \prod_{e:z(e) = 0} 1- \lambda(e) \right],$$
or equivalently,
$$\prob(\Lambda = z) = Z_\on{subs}^{-1} 2^{-\#V}  
  \left[ \prod_e 1 + \lambda(e) \right]
  2^{\#\cluster(z) } 
                       \left[ \prod_{e:z(e) = 1} \frac{2 \lambda(e)}
 {1 + \lambda(e)} \right] 
                       \left[ \prod_{e:z(e) = 0} \frac{1 - \lambda(e)}
 {1 + \lambda(e)} \right].$$
It is easy to verify from (\ref{EQN:lambda_definition}) 
and (\ref{EQN:p_definition}) 
that $2 \lambda(e) / (1 + \lambda(e)) = p(e)$
and $(1 - \lambda(e)) / (1 + \lambda(e)) = 1 - p(e)$, 
which shows that the two rightmost sets of brackets factors equal
$w_\on{rc}(z).$  Therefore 
$Z^{-1}_\on{subs} 2^{-\#V} \prod_e (1 + \lambda(e))  
  = Z^{-1}_\on{rc}.$  
Using $(1 + \lambda(e))^{-1} = (1/2)(1 + \exp(-2\beta(e)))$
completes the proof.
\end{proof}

\section{Random cluster to subgraphs}
\label{SEC:rc2subs}

This section shows how to take a random cluster draw, and together
with a number of Bernoullis equal to $\#E - \#V$, create a draw
from the subgraphs model.
Swendsen-Wang~\cite{swendsenw1987} employed the relationship between
random clusters and spins to devise a Markov chain that is very fast
for some values of the parameters.  From a spins draw, generate a 
random cluster index, then from the index, generate a new spins draw.

With the ability to move between a random cluster draw and a subgraphs
draw, a new Swendsen-Wang style Markov chain becomes available for use:
from the random cluster model, draw a subgraphs model, then from the 
subgraphs model, draw a random cluster model.
Algorithm~\ref{ALG:randomcluster2subgraphs} 
shows how to move from a random cluster model to a subgraphs model.

\begin{algorithm}
\caption{{\tt Random Cluster to subgraphs}}
\label{ALG:randomcluster2subgraphs}
\begin{algorithmic}[1]
\INPUT  parameter $p(e)$, $Z \sim \pirc(\cdot;p)$
\OUTPUT $Y \sim \pisubs$
  \STATE $F_E \leftarrow \textrm{ a maximal forest using edges } e 
    \textrm{ with } Z(e) = 1.$
  \STATE $D \leftarrow \emptyset$ 
  \FOR {all edges $e \in E \setminus F_E$}
    \STATE {\bf draw} $Y(e) \leftarrow \bern(Z(e)/2)$, 
                      $D \leftarrow D \cup \{e\}$
  \ENDFOR 
  \WHILE{$F_E \neq \emptyset$}
    \STATE $i \leftarrow \textrm{any leaf in the forest } F_E$ 
    \STATE $j \leftarrow \textrm{the node such that } \{i,j\} \in F_E$
    \STATE $Y(\{i,j\}) \leftarrow \ind(\degree(i,Y(D)) \textrm{ is odd })$
    \STATE $D \leftarrow D \cup \{i,j\}$, 
           $F_E \leftarrow F_E \setminus \{i,j\}$
  \ENDWHILE
\end{algorithmic}
\end{algorithm}

\begin{thm} Algorithm~\ref{ALG:randomcluster2subgraphs} returns a draw
from the subgraphs distribution where parameters $\lambda$, $\beta$ and $p$
are related by equations (\ref{EQN:lambda_definition}) and 
(\ref{EQN:p_definition}).
\end{thm}

\begin{proof} 
It was shown in 
Theorem~\ref{THM:subs2rc} that if $Y \sim \pi_\on{subs}(\cdot;\lambda)$,
and then $Z|Y'$ is drawn as an index for the component of the mixture,
then $Z \sim \pi_\on{rc}(\cdot;p)$ where $p$ and $\lambda$ are
related by (\ref{EQN:lambda_definition}) and 
(\ref{EQN:p_definition}).  So as in the discussion in Section~\ref{SEC:theory},
if $Z$ is drawn first from $\pi_\on{rc}(\cdot;p)$, and then 
$Y|Z \sim \pi_{\on{subs}}(\cdot;Z)$, then 
$Y \sim \pi_\on{subs}(\cdot;\lambda).$

Let $F_E$ be a maximal spanning forest in the graph using
edges with $Z(e) = 1$.  
Any edge not in $F_E$ with $Z(e) = 0$ must have $Y(e) = 0$ in order for the 
configuration to have positive weight.  So 
let $e$ be an edge such that $Z(e) = 1$ in $E \setminus F_E$.

Since $F_E$ is a maximal forest, the endpoints of $e$ must be connected
by edges in $F_E$.  Call $e$ plus these connecting edges $L$.  
Let $x$ be any configuration with $x(e) = 0$, and 
$f(x)$ be a new configuration constructed as follows.  For all 
$e' \in L$, let $f(e') = 1 - x(e').$  For all $e' \notin L$, let
$f(e') = x(e').$  This `flips' the value of $x(e')$ along 
all the edges of the cycle $e \cup L$.

Since the edges are flipped along a cycle, the parity of each node
is unchanged.
This map $f$ is 1-1 and onto, and so the partition function conditioned
on $y(e) = 0$ equals that conditioned on $y(e) = 1$.  In other words,
there is a exactly a 1/2 chance that $Y(e) = 1$ for a draw $Y \sim \pisubs.$

This holds even when conditioned on the values of all other edges in
$E \setminus (F_E \cup \{e\}),$ since only edges in $F_E$ plus $e$ 
were used to construct $L$.  Combining the $Z(e) = 0$ and $Z(e) = 1$
cases, the result is that $Y(e) \sim \bern(Z(e) / 2)$, even when conditioned
on the value of $Y(e')$ for all other edges $e' \neq e$ not in $F_E$.
Hence lines 3 through 5 of the algorithm
are correct.

Now consider edge $e = \{i,j\}$ where $i$ is a leaf of $F_E$.  Then
$Y(e)$ must be chosen in such a way that the degree of $i$ in $Y$
is even.  Line 9 accomplishes this task.  This edge has been assigned
a value, and can now be removed from $F_E$.  As long as $F_E$ is 
nonempty, there will be at least two leaves left, so the process can
be continued until the edges of $F_E$ are all assigned values by $Y$.
\end{proof}

Note that the maximal forest $F_E$ can be constructed using either breadth
first or depth first search.  Either way, reversing the order in which 
nodes were added to the forest provides a way of recovering the leaves
one by one with an overall running time that is linear in the size
of the graph.

\section{Further notes}
\label{SEC:discuss}

\subsection{Proof of Theorem~\ref{THM:relate}}

The reduction from subgraphs to random clusters as shown in 
Theorem~\ref{THM:subs2rc} together with the reduction from random
clusters to spins provides an immediate proof of Theorem~\ref{THM:relate}.

\begin{proof}[Proof of Theorem~\ref{THM:relate}]
Let $Z$ be a random cluster draw, and $X$ the spins draw that is obtained
by uniformly at random choosing spin up or down for each cluster in $Z$.
For a spins configuration $x$, let $S(x)$ denote the set of random cluster
configurations that are consistent with $x$, so $z \in S(x)$ means for all
$\{i,j\} \in E$, $x(i) = x(j)$ implies $z(\{i,j\}) = 1$.
Then
\begin{eqnarray*}
\prob(X = x) & = & \sum_{z \in S(x)} \prob(X = x|Z = z)\prob(Z = z) \\
  & = & \sum_{z \in S(x)} (1/2)^{\# \cluster(z)} 
          \left[\prod_{e:z(e) = 1} p(e)\right] 
  \left[ \prod_{e:z(e) = 0} (1 - p(e)) \right] 
  2^{\#\cluster(z)} Z_\on{rc}^{-1} \\
 & = & Z_\on{rc}^{-1} \prod_e \sum_{z(e) \leq \ind(x(i) = x(j))} p(e) z(e) 
         + (1 - p(e))(1 - z(e)) \\
 & = & Z_\on{rc}^{-1} \prod_e 1 \cdot \ind(x(i) = x(j)) 
         + (1 - p(e)) \cdot \ind(x(i) \neq x(j)).
\end{eqnarray*}
Now $1 - p(e) = 0$ if $\beta(e) = \infty$.  When $\beta(e) < \infty$,
$1 - p(e) = \exp(-2\beta(e))$.  
This means that for all $\beta(e)$:
\[
\ind(x(i) = x(j)) + (1 - p(e)) \ind(x(i) \neq x(j)) = 
 f(x(i),x(j)) \exp(-\beta(e)) 
\]
where $\exp(-\infty)$ is taken to be 0, and $f$ is as in 
(\ref{EQN:f_definition}).  Hence
\begin{eqnarray*}
\prob(X = x) & = & 
 Z_{\on{rc}}^{-1} \prod_{e} [f(x(i),x(j)) \exp(-\beta(e))],
\end{eqnarray*}
which means that 
$Z_\on{spins} = Z_\on{rc} \prod_e \exp(\beta(e)).$

Combining with Theorem~\ref{THM:subs2rc} yields
\begin{eqnarray*}
Z_\on{spins} & = & Z_\on{subs} 2^{\#V - \#E} \left[ 
  \prod_e (1 + \exp(-2\beta(e))) \right]
  \left[ \prod_e \exp(\beta(e)) \right] \\
 & = & Z_\on{subs} 2^{\#V} 
  \prod_e \cosh \beta(e),
\end{eqnarray*}
as desired.
\end{proof}

\subsection{Perfect simulation of subgraphs}
\label{SBS:perfectsim}

One of the original applications of the coupling from the past 
algorithm~\cite{proppw1996} was to generate samples perfectly from
the random cluster model.  Given the reduction of the previous section,
this immediately gives the first perfect simulation method for the
subgraphs world, which could prove useful in studying the model.

\bibliographystyle{plain}

\begin{thebibliography}{10}

\bibitem{besag1974}
J.~Besag.
\newblock Spatial interaction and the statistical analysis of lattice systems
  (with discussion).
\newblock {\em J. R. Statist. Soc. B}, 36:192--236, 1974.

\bibitem{fortuink1972}
C.M. Fortuin and P.W. Kasteleyn.
\newblock On the random cluster model {I}: Introduction and relation to other
  models.
\newblock {\em Phys.}, 57:536--564, 1972.

\bibitem{jerrums1993}
M.~Jerrum and A.~Sinclair.
\newblock Polynomial-time approximation algorithms for the {I}sing model.
\newblock {\em SIAM J. Comput.}, 22:1087--1116, 1993.

\bibitem{jerrumvv1986}
M.~Jerrum, L.~Valiant, and V.~Vazirani.
\newblock Random generation of combinatorial structures from a uniform
  distribution.
\newblock {\em Theoret. Comput. Sci.}, 43:169--188, 1986.

\bibitem{newellm1953}
G.F. Newell and E.W. Montroll.
\newblock On the theory of the {I}sing model of ferromagnetism.
\newblock {\em Rev. Modern Phys.}, 25:353--389, 1953.

\bibitem{proppw1996}
J.~G. Propp and D.~B. Wilson.
\newblock Exact sampling with coupled {M}arkov chains and applications to
  statistical mechanics.
\newblock {\em Random Structures Algorithms}, 9(1--2):223--252, 1996.

\bibitem{randallw1999}
D.~Randall and D.~Wilson.
\newblock Sampling spin configurations of an {I}sing system.
\newblock In {\em Proc. 10th ACM-SIAM Sympos. on Discrete Algorithms}, pages
  959--960, 1999.

\bibitem{simon1993}
B.~Simon.
\newblock {\em The Statistical Mechanics of Lattice Gasses}, volume~1.
\newblock Princeton University Press, 1993.

\bibitem{swendsenw1987}
R.~Swendsen and J-S. Wang.
\newblock Non-universal critical dynamics in {M}onte {C}arlo simulation.
\newblock {\em Phys. Rev. Lett.}, 58:86--88, 1987.

\bibitem{vanderwaerden1941}
B.~L. van~der {W}aerden.
\newblock Die lang {R}eichweite der regelm{\"a}\ss{}igen {A}tomanordnung in
  {M}ischkristallen.
\newblock {\em Z. Physik}, 118:473--488, 1941.

\end{thebibliography}

\end{document}